\RequirePackage{fix-cm}
\documentclass[smallextended]{svjour3}       
\smartqed  
\usepackage{graphicx}
\usepackage{amsmath}
\usepackage{amsfonts}
  \newcommand{\lip}{\mathtt{1Lip}}
  \newcommand{\dH}{d_{\mathrm{H}}}
  \newcommand{\dgh}{d_{\mathrm{GH}}}
  
    \newtheorem{observation}{Observation}
  
  \DeclareMathOperator{\dis}{dis}

\begin{document}

\title{Approximate Capture \\ in Gromov--Hausdorff Closed Spaces \thanks{This study was supported by the Russian Science Foundation (project no.~17-11-01093).}
}

\titlerunning{Approximate Capture in Gromov--Hausdorff Closed Spaces}        

\author{Olga Yufereva}


\institute{O. Yufereva \at
              Krasovskii Institute of Mathematics and Mechanics \\
              \email{olga.o.yufereva@gmail.com}     }

\date{Received: date / Accepted: date}

\maketitle

\begin{abstract}
We consider the  Lion and Man game, i.e.,  a two-person pursuit-evasion game with equal players' top speeds. We assume that  capture radius is positive and chosen in~advance. The main aim of the paper is describing pursuer's winning strategies in~general compact metric spaces that are close  to the given one in the sense of Gromov--Hausdorff distance.  
We prove that the existence of  $\alpha$-capture by a time $T$ in~one compact geodesic space implies the existence of  $\bigl(\alpha + (20T +8)\sqrt{\varepsilon}\bigr)$-capture by this time $T$ in~any compact geodesic space that is $\varepsilon$-close to the given space.
 It means that  capture radii (in~a~nearby spaces) tends to the given one as the distance between spaces tends to zero. Thus, this result justifies calculations on graphs instead of complicated spaces.
\end{abstract}
\section*{Introduction}

The Lion and Man game is linked with different areas of mathematics and is applicable, e.g., to studying Browian motions~\cite{bramson2013shy} and behaviour of pursuit curves on CAT(k) spaces~\cite{alexander,jun} and to getting a new geometric properties  \cite{kohlenbach2018quantitative}.
In brief,the Lion and Man game is a two-person game that assumes players to have the same capabilities (typically, the same top speed). We also consider  positive capture radii and 
 assume that players move along 1-Lipschitz curves only. 
These assumptions allow us to expand  consideration to general metric spaces, beyond Euclidean cases.

This paper focuses on the possibility of transporting this game from one metric space to another, `nearby' metric space.
 More precisely, we build the strategies transfer that provides the robustness of capture radius. 
 This approach significantly differs from the classic conception of Krasovskii and Subbotin as we may not use smooth structure. 
 In particular, the classical  works \cite{kras1,kras2}, as well as the stochastic generalizations \cite{averboukh,kras3}, use derivations of Lyapunov functions and guidance methods, whereas 
we introduce  similar constructions dealing with general compact metric spaces. 
More exactly, we study the robustness of capture radius with respect to Gromov--Hausdorff distance, i.e., the distance between compact metric spaces.

Our main result states capture robustness in the following form: \\ {\it Let two compact geodesic spaces be such that the Gromov--Hausdorff (or Hausdorff) 
distance between them  is not greater than $\varepsilon.$ Existence of  $\alpha$-capture by a time $T$ in one space implies  existence of  $\bigl(\alpha + (20T +8)\sqrt{\varepsilon}\bigr)$-capture by the time $T$ in the other space.} \\
This result justifies replacing a space by a sequence of finite graphs: namely, we can approximate the space by such a~sequence, consider the capture radii in these graphs and make a conclusion about the capture radius in the first space. 
 We return to the applications in Section \ref{sec-5-res}.

Note that this theorem uses an upper bound on time, whereas usually, the infinite horizon is considered (see e.g. \cite{alexander,kohlenbach2018quantitative,yufereva}). 
Note that the existence of the upper bound is proved in \cite{bramson2013shy} for  bounded CAT(0) domains with several extra restrictions.
In addition, there are estimates for upper and lower bounds on capture time for the disk in the paper~\cite{alonso}. 
However, as far as we know, there are no similar estimates in general metric spaces.

The paper is organized as follows. Section \ref{sec-2-st} introduces game and geometry definitions, Section~\ref{sec-5-res} contains the theorem and a little discussion. Finally,  Section~\ref{sec-proof} provides the proofs.

\section{Preliminaries}
\label{sec-2-st}

\subsection{Lion and Man game}

We consider the Lion and Man problem, where Lion is a pursuer and Man is an evader. We assume that  both players  move in a metric space $(X, \rho)$ and denote by $L(\cdot)$ and $M(\cdot)$ the~trajectories of Lion and Man respectively. Notice that $L(\cdot)$ and $M(\cdot)$ are 1-Lipschitz functions of time to the space $X.$ 
We consider the game from Lion's standpoint. Hence, Lion uses  non-anticipative (or even stepwise) strategies against arbitrary possible movements of Man. By non-anticipative and stepwise strategies we mean the following.  

Let us denote by $\lip(X)$  the set of all 1-Lipschitz curves from $\mathbb{R_{+}}$ to a set $X$ with respect to the metric $\rho$ on this set $X.$ It is the set of each player's admissible trajectories. 
A map $\mathfrak{s}_{L_0} \colon \lip(X) \to \lip(X)$ is called Lion's {\it non-anticipative} strategy (with Lion's initial position $L_0$) if
it satisfies the equality $\mathfrak{s}_{L_0}(M)(0) = L_0$ for all $M\in \lip(X)$ and 
the following implication holds true: for admissible trajectories  $M_1, M_2 \in \lip(X)$ and a number $\tau\geq 0,$ if 
$$M_1(t) = M_2(t) \quad \forall t \in [0, \tau],$$
then
$$\mathfrak{s}_{L_0}(M_1)(t) = \mathfrak{s}_{L_0}(M_2)(t) \quad \forall t \in [0, \tau]. $$
A particular case of non-anticipative strategy is a stepwise strategy. For a positive number $\beta,$ we say that a map~$\mathfrak{s}_{L_0}$ is called Lion's {\it $\beta$-stepwise} strategy (with Lion's initial position $L_0$) if 
it satisfies the equality $\mathfrak{s}_{L_0}(M)(0) = L_0$ for all $M\in \lip(X)$ and 
the following implication holds true: for  admissible trajectories $M_1, M_2 \in \lip(X)$ and a number $n\in \mathbb{N},$ if 
\begin{equation*}
M_1(t) = M_2(t) \quad \forall t \in [0, n\beta],
\end{equation*}
then
$$\mathfrak{s}_{L_0}(M_1)(t) = \mathfrak{s}_{L_0}(M_2)(t) \quad \forall t \in [0, (n+1)\beta]. $$
\begin{definition}
 If $(X, \rho)$ is a compact metric space and $\alpha$ is  a positive number, then by the phrase `there are $(X, \rho, \alpha, T)$-winning strategies' let us denote the fact that, for any Lion's initial position $L_0,$ he has a non-anticipative strategy $\mathfrak{s}_{L_0} \colon \lip(X) \to \lip(X) $ such that
$$\mathfrak{s}_{L_0}(M)(0)=L_0,$$
 $$\inf\limits_{t\in [0,T]} \rho\bigl(M(t), \mathfrak{s}_{L_0}(M)(t)\bigr)\leq\alpha$$  for all Man's movements $M(\cdot)\in \lip(X).$
\end{definition}

\subsection{Basic notation}

Let us provide definitions and notations under the assumption that $(X,\rho_X)$ and $(Y,\rho_Y)$ are  metric spaces, map $f$ is from $X$ to $Y,$ and $a$ is a positive number.
\begin{itemize}
\item $dom(f)$ denotes the domain of the map $f;$
\item $\mathbb{R}_{+} = [0, +\infty);$
\item $ \lceil a\rceil$ is the ceiling function of the number $a;$ 
\item if $S\subset X,$ then $f[S]$  denotes the set $\{f(x)\mid x\in S\};$
\item a set $S\subset X$ is called an {\it $\varepsilon$-net} of the metric space $(X, \rho)$ if, for any point $x\in X,$ there is a point $s\in S$ such that $\rho(x, s)\leq \varepsilon;$
\item a map $g$ from a closed interval $[0, l] \subset \mathbb{R}$ to $X$ is called a {\it geodesic path} iff  $g(0) = x, \ g(l) = y$ and $\rho_{X}(g(t), g(t')) = |t-t'|$ for all $t, t' \in [0, l]$ (see \cite[I.1.3]{bh});
 \item $(X, \rho_X)$ is said to be a {\it geodesic space} if every two points $x$ and $y$ in $X$ are joined by a geodesic path;
 \item 
and if, for every pair of points $x_1, x_2 \in X,$ the distance $\rho_X(x_1, x_2)$ is equal to the infimum of the length of rectifiable curves joining these points, then $(X, \rho_X)$ is called a{\it~length~space}, otherwise known as an {\it inner metric space} (see \cite[I.3.3]{bh}).
\end{itemize}

\subsection{Distance between metric spaces}

Definition of the Gromov--Hausdorff distance is fundamental in this paper, but one can imagine the Hausdorff distance instead of the Gromov--Hausdorff one in the case when  all considered spaces are subspaces of one ambient space. 
The distance was proposed by Gromov and Edwards independently (see \cite{tuzhilin2016invented}). Let us give preference to Gromov's approach and use the following definition of the Gromov--Hausdorff distance borrowed from \cite[Def. 7.3.10.]{burago}:
\begin{definition}
Let $X$ and $Y$ be metric spaces. The Gromov--Hausdorff distance between them, denoted by $\dgh(X, Y),$ is defined by the following 
relation. For an $r > 0,$ we have $\dgh(X, Y)<r$  if and only if there exist a metric space $Z$ and subspaces $X'$ and $Y'$ of it that are isometric to $X$ and $Y,$ respectively, and such that $\dH(X', Y') < r.$ In other words, $\dgh(X,Y)$ is the infimum of positive $r$ for which the above $Z, X',$ and $Y'$ exist. Here $\dH$ denotes the Hausdorff distance between subsets of $Z.$ 
\end{definition}
There exist other equivalent definitions, they are helpful to prove that this distance is a {\it metric} on the (continual) set of compact metric spaces. But, these definitions do not reflect the main idea so good. One can find both historic remarks and the list of important properties of the {\it Gromov--Hausdorff space} in the paper \cite{tuzhilin2016invented}. 
Let us introduce some related definitions and properties. 

\begin{definition}
Let $(X, \rho_X)$ and $(Y, \rho_Y)$ be metric spaces and let $f \colon X \to Y$ be an 
arbitrary map. The {\it distortion of $f$} (denoted by $\dis f$) is defined as  
$$ \dis f = \sup\limits_{x_1,x_2 \in X} |\rho_{X}\bigl(x_1, x_2\bigr)- \rho_{Y}\bigl(f(x_1), f(x_2)\bigr)|.$$
\end{definition}
\begin{definition}
Let $X$ and $\tilde{X}$ be metric spaces and let $\varepsilon > 0.$ A map $h \colon X \to \tilde{X}$ is called an{\it~$\varepsilon$-isometry} from $X$ to $\tilde{X}$ if $\dis h \leq \varepsilon$ and $h[X]$ is an $\varepsilon$-net in~$\tilde{X}.$ 
\end{definition}
\begin{lemma}[\cite{burago} Cor.7.3.28] \label{cor-bur}
Let $X$ and $\tilde{X}$ be two metric spaces and let $\varepsilon > 0.$ Then, \\
1. if $\dgh (X, \tilde{X}) \leq \varepsilon,$ then there exists a $2\varepsilon$-isometry from $X$ to $\tilde{X};$ \\
2. if there exists an $\varepsilon$-isometry from $X$ to $\tilde{X},$ then $\dgh (X, \tilde{X}) \leq 2\varepsilon.$ 
\end{lemma}
\section{Results}
\label{sec-5-res}

Let us recall that, by virtue of Hopf--Rinow theorem, a compact geodesic space is a~compact length space and vice versa. Hence, the distance between every two points of a compact geodesic space is  given by the infimum of the lengths of rectifiable paths joining these points. We deal with these spaces because of convenience and since any compact metric space can be reparametrized into a length space if players' trajectories form a length structure (for details, see \cite[Chapter 2]{burago}).

\begin{theorem} \label{th-x-eps} 
Let $(\tilde{X}, \tilde{\rho})$ and $(X, \rho)$ be compact geodesic spaces, let $T>0, $ and let $\alpha\in (0,1).$ If $d_{GH}(\tilde{X}, X)\leq\varepsilon\in (0, \alpha^2),$ then the existence of  $(\tilde{X}, \tilde{\rho}, \alpha, T)$-winning strategies  implies the existence of  $(X, \rho, \alpha + (20T +8)\sqrt{\varepsilon}, T)$-winning strategies.
\end{theorem}
\begin{remark}
Since the Gromov--Hausdorff distance is never greater than any Hausdorff distance (among possible isometric embeddings),  it suffices to check the  Hausdorff distance  inequality $\dH(\tilde{X}, X)\leq\varepsilon$  instead of the Gromov--Hausdorff one.
\end{remark}
\begin{remark}
Although we consider non-anticipative strategies, to prove the theorem, we construct $\sqrt{\varepsilon}$-stepwise strategies. Hence, all statements hold for the so-called discrete-time Lion and Man games.
\end{remark}
The following corollary is trivial, however, it provides a necessary condition of capture, which is quite rare.
\begin{corollary}
If  there are no  $(\tilde{X}, \tilde{\rho}, \alpha + (20T +8)\sqrt{\varepsilon}, T)$-winning strategies for $d_{GH}(\tilde{X}, X)\leq\varepsilon\in (0, \alpha^2)$, then there are no  $(X, \rho, \alpha, \  T)$-winning strategies.
\end{corollary}

Moreover, one can combine these results  with the fact that, for any compact geodesic (as well as length) space~$(X, \rho),$ there exists a sequence of  finite metric  graphs  $(G_n, \rho_n)$  such that $\dgh(X, G_n)\rightarrow 0$ as $n\rightarrow \infty$ (see \cite[Proposition~7.5.5.]{burago}).  The strict definition of finite metric graph is not short (see \cite[Def. 3.2.11.]{burago} or \cite[I.1.9]{bh}), although the notion  is quite intuitive. Roughly speaking, a finite metric  graph is a metric space represented as a finite number of vertices and a finite number of segments connecting some of these vertices. 
 So this encourages one to try to check $\alpha$-capture in some approximating spaces before doing this  in the approximated space. 
  In this way, Theorem \ref{th-x-eps} yields the following corollary.
\begin{corollary}
\label{cor-graph}
Let  $(X, \rho)$ be a compact geodesic space and let $\{(G_n, \rho_n)\}_{n \in \mathbb{N}}$ be a sequence of finite metric  graphs. 
If  $d_{GH}(X, G_n)\rightarrow 0$ as $n\rightarrow \infty,$ then the following statements are equivalent: 
\begin{enumerate}
\item for all $\tilde{\alpha}>\alpha$ there exist $(X, \rho, \tilde{\alpha}, T)$-winning strategies;
\item for each $n\in \mathbb{N}$ there exist $(G_n, \rho_n, \alpha_n, T)$-winning strategies and $\alpha_n \rightarrow \alpha$ as $n\rightarrow\infty.$
\end{enumerate}
\end{corollary}
\begin{remark}
There is a constructive method to build such graphs.
\end{remark}
Thus, investigating the game in finite metric graphs should be useful. In particular, if we knew `good enough' pursuer's strategies in graphs, we would easily construct `good enough' strategies in any compact metric space. %
 
\section{Proofs}
\label{sec-proof}
We begin by constructing  $\beta$-pursuing curves (Subsection \ref{ssec-sim-pur}) and $\varepsilon$-chaining relations between spaces (Subsection \ref{ssec-ch}). These subsections are  auxiliary for Subsection \ref{ssec-proof}, which is directed towards the target proof. 
 
\subsection{Pursuit}  
\label{ssec-sim-pur}

We need to use the stepwise procedure of simple (greedy) pursuit in the following form. 
 \begin{definition}
 \label{def-simple-pursuit}
Let $N$ be a natural number, let $\beta>0,$ and let $(X, \rho)$ be a compact geodesic space.  We say that a 1-Lipschitz curve $\hat{L} \colon [0, N\beta] \to X$ $\beta$-pursues 
a tuple $\{\hat{M}_i\}_{i=\overline{0,..,N}}, $ where $\hat{M}_i \in X$  iff, 
for each $i = \overline{0,..,N-1},$
\begin{enumerate}
\item the condition  $\rho\bigl(\hat{L}(i\beta), \hat{M}_i \bigr)> \beta$ implies
\begin{equation*}
\begin{split}
\rho\bigl(\hat{L}(i \beta), \hat{L}((i+1)\beta)\bigr) &=  \beta,
\\
\rho\bigl(\hat{L}(i\beta), \hat{L}((i+1)\beta)\bigr)+ \rho\bigl(\hat{L}((i+1)\beta), \hat{M}_i\bigr) &= \rho\bigl(\hat{L}(i\beta), \hat{M}_i \bigr); 
\end{split}
\end{equation*}
\item the opposite condition $\rho\bigl(\hat{L}(i\beta), \hat{M}_i\bigr)\leq \beta$ implies
\begin{equation*}
\hat{L}\bigl((i+1)\beta\bigr) =  \hat{M}_i.
\end{equation*}
\end{enumerate}
\end{definition}
\begin{remark}
\label{rem-tr-eq}
Whatever the numbers $\beta$ and $N$ and points $\hat{L}_0$ and $\{\hat{M}_i\}_{i=\overline{0,..,N-1}}$ are, there is a curve that $\beta$-pursues the tuple $\{\hat{M}_i\}_{i=\overline{0,..,N-1}}$ and satisfies $\hat{L}(0) = \hat{L}_0.$ 
Indeed, the equalities 
of Item 1 of Definition \ref{def-simple-pursuit} mean that the restriction $\hat{L}|_{[i\beta, (i+1)\beta]}$ is a geodesic path and the restriction of a geodesic path between $\hat{L}(i\beta)$ and $\hat{M}(i\beta).$ 
Since we assume that $(X, \rho)$ is a compact geodesic space, then such a geodesic path exists, but may not be unique.
\end{remark}
\begin{observation} 
\label{obs-tuple}
Whatever numbers $\beta$ and $N$ and pursuer's initial position are, if a tuple  $\{\hat{M}_i\}_{i=\overline{0,..,N}} $ is such that each $\hat{M}_i$ depends on the restriction $M|_{[0, i\beta]}$ only (for $i=\overline{0,..,N}$), then the pursuer has a $\beta$-stepwise strategy that provides that the pursuer's trajectory $\beta$-pursues the tuple $\{\hat{M}_i\}_{i=\overline{0,..,N}}. $
\end{observation}
\begin{lemma}
 \label{lem-simple-pursuit}
Let $(X, \rho)$ be a compact geodesic space. If a 1-Lipschitz curve $\hat{L} \colon [0, N\beta] \to X$ $\beta$-pursues a tuple $\{\hat{M}_i\}_{i=\overline{0,..,N}}$ satisfying 
$$\rho\bigl(\hat{L}(i\beta), \hat{M}_i\bigr)\leq \beta(1+\delta) \qquad \forall i=\overline{0,..,N-1},$$
 then
\begin{equation*}
\rho\bigl(\hat{L}(N\beta), \hat{M}_N\bigr)\leq \beta + N\delta\beta + \rho\bigl(\hat{L}(0), \hat{M}_0\bigr).
\end{equation*}
\end{lemma}
\begin{proof}
Let us prove this by induction on $N.$ The  induction basis is trivial: $\rho\bigl(\hat{L}(0), \hat{M}_0\bigr)\leq \beta + 0\delta\beta + \rho\bigl(\hat{L}(0), \hat{M}_0\bigr).$ Further,  the inequality hypothesis is
\begin{equation}
\label{eq-ind}
\rho\bigl(\hat{L}(i\beta), \hat{M}_i \bigr)\leq \beta + i\delta\beta + \rho\bigl(\hat{L}(0), \hat{M}_0\bigr).
\end{equation}
Let us show that this inequality still holds for $(i+1)$ instead of $i.$
Consider two cases. First, if we have  $\rho\bigl(\hat{L}(i\beta), \hat{M}_i \bigr)\leq\beta,$  then $\hat{L}((i+1)\beta) = \hat{M}_i$ by~Item~2 of Definition~\ref{def-simple-pursuit}  and, as a corollary,
\begin{eqnarray*}
\rho\bigl(\hat{L}((i+1)\beta), \hat{M}_{i+1}\bigr)&\leq& \rho(\hat{M}_i, \hat{M}_{i+1} ) \\
 &\leq& \beta(1+\delta) \leq \beta + (i+1)\delta\beta + \rho\bigl(\hat{L}(0), \hat{M}_0\bigr).
\end{eqnarray*}

Second, if we have $\rho\bigl(\hat{L}(i\beta), \hat{M}_i\bigr)>\beta,$  then the following holds:
$$\begin{array}{lr}
 \quad  \rho\bigl(\hat{L}((i+1)\beta), \hat{M}_{i+1}\bigr) \quad   &\mbox{ by triangle equality} \\
\leq   \rho\bigl(\hat{L}((i+1)\beta), \hat{M}_i\bigr) + \rho\bigl(\hat{M}_i, \hat{M}_{i+1}\bigr)
&\mbox{ by the hypothesis of the lemma } \\
\leq  \rho\bigl(\hat{L}((i+1)\beta), \hat{M}_i\bigr)  + \beta + \delta\beta 
  &\mbox{ by Item~1 of Definition~\ref{def-simple-pursuit}}  \\
= \rho\bigl(\hat{L}(i\beta), \hat{M}_i\bigr) - \rho\bigl(\hat{L}(i\beta), \hat{L}((i+1)\beta)\bigr)  + \beta + \delta\beta
 &\mbox{ by Item 1 of Definition~\ref{def-simple-pursuit}}  \\
= \rho\bigl(\hat{L}(i\beta), \hat{M}_i \bigr) + \delta\beta 
&\mbox{ by the induction assumption (\ref{eq-ind})} \\
\leq  \beta + (i+1)\delta\beta + \rho\bigl(\hat{L}(0), \hat{M}_0\bigr).	  
& 
\end{array}$$
\end{proof}

\subsection{Chaining} 
\label{ssec-ch}

\begin{definition}
\label{def-ch}
Let $(X, \rho)$ and $(\tilde{X}, \tilde{\rho})$ be metric spaces 
and let  
$$f \colon \tilde{X} \to X, \quad \tilde{f}\colon X \to \tilde{X}. $$
We say that the pair of maps $(f, \tilde{f})$ is an {\it   $\varepsilon$-chaining between $X$ and $\tilde{X}$} if there exist
\begin{enumerate}
\item  finite $\varepsilon$-nets $R_{X}$ and $R_{\tilde{X}}$ in $X$ and $\tilde{X}$ respectively;
\item bijections $h \colon R_{\tilde{X}} \to R_{X}$  and $\tilde{h} = h^{-1}\colon R_{\tilde{X}}\to R_X$
	such that $$\dis h \leq \varepsilon/2, \quad \dis \tilde{h} \leq \varepsilon/2;$$
\item maps $g\colon X\to R_{X}$ and $\tilde{g}\colon \tilde{X}\to R_{\tilde{X}}$ 
	such that 
\begin{eqnarray*}
	g(x) \in \arg\!\min\limits_{y\in R_X} \rho(x, y),& \quad &\tilde{f} = \tilde{h} \circ g, \\
\tilde{g}(\tilde{x}) \in \arg\!\min\limits_{\tilde{y}\in R_{\tilde{X}}} \rho(\tilde{x}, \tilde{y}),& \quad &f  = h \circ \tilde{g}. 
\end{eqnarray*}
\end{enumerate}
\end{definition}
\begin{remark}
\label{rem-ch}
If the pair $(f, \tilde{f})$ is an $\varepsilon$-chaining between spaces $X$ and $\tilde{X}$ and maps $g$ and $\tilde{g}$ are defined according to Definition \ref{def-ch}, then $\rho\bigl(x,g(x)\bigr)\leq \varepsilon$ for all $x\in X$ and $\tilde{\rho}\bigl(\tilde{x},\tilde{g}(\tilde{x})\bigr)\leq \varepsilon$ for all  $\tilde{x}\in \tilde{X}.$
\end{remark}
\begin{lemma}
\label{lem-ch-exist}
If $(X, \rho)$ and $(\tilde{X}, \tilde{\rho})$ are compact metric spaces and $\dgh(X, \tilde{X})\leq \varepsilon,$ then there exists $4\varepsilon$-chaining $(f, \tilde{f})$ between $X$ and $\tilde{X.}$
\end{lemma}
\begin{proof}	
We can pick a $2\varepsilon$-isometry $\hat{h}\colon X\to \tilde{X}$ in accordance with Lemma~\ref{cor-bur}. This implies that the set $\hat{h}[X]$ is a $2\varepsilon$-net in $\tilde{X},$ and, by virtue of compactness of $\tilde{X},$ we can select a finite subset of this $2\varepsilon$-net that is a $2\varepsilon$-net too; let us denote it by $R_{\tilde{X}}.$ For each $y\in R_{\tilde{X}},$ let us pick a point  $x_y$ in the non-empty set $\hat{h}^{-1}(y)$ and then consider $R_{X} = \{x_y \mid y\in R_{\tilde{X}}\}.$ Note that $|R_{\tilde{X}}| = |R_{X}|.$ 

Let us show that $R_X$ is a $4\varepsilon$-net in $X.$ Assume the converse, i.e., that there is a point  $x\in X$ such that $\rho(x, x_y)>4\varepsilon$ for all $x_y \in R_X.$ This yields that, for all  $y\in R_{\tilde{x}},$
\begin{equation}
\label{eq-lem-ch}
2\varepsilon = \dis\hat{h}\geq |\rho(x,x_y) - \tilde{\rho}\bigl(\hat{h}(x),\hat{h}(x_y)\bigr)| = |\rho(x,x_y) - \tilde{\rho}\bigl(\hat{h}(x),y\bigr)|. 
\end{equation}
Since $R_{\tilde{x}}$ is a $2\varepsilon$-net in $\tilde{X},$ there is $y\in R_{\tilde{x}}$ such that $\tilde{\rho}\bigl(\hat{h}(x),y\bigr)\leq 2\varepsilon.$ By this and by $\rho(x, x_y)>4\varepsilon,$ we obtain 
$$|\rho(x,x_y) - \tilde{\rho}\bigl(\hat{h}(x),y\bigr)|>4\varepsilon-2\varepsilon>2\varepsilon$$
and this contradicts  (\ref{eq-lem-ch}). Thus, $R_X$ is a $4\varepsilon$-net in $X.$

Further, note that the restriction $h = \hat{h}|_{R_X}$ is bijective and $$\dis h\leq \dis \hat{h}\leq 2\varepsilon.$$ Then, $\tilde{h} = h^{-1}$ is also bijective and satisfies $\dis \tilde{h}\leq 2\varepsilon.$ Since the sets $R_X$ and $R_{\tilde{X}}$ are finite, the maps $g$ and $\tilde{g}$ exist. Thus, the maps  $f  = h \circ \tilde{g}$ and $\tilde{f} = \tilde{h} \circ g $ are such that the pair $(f, \tilde{f})$ is a $4\varepsilon$-chaining between $X$ and $\tilde{X}.$
\end{proof}
\begin{lemma}
\label{lem-f-tilf}
If a pair $(f, \tilde{f})$ is a $4\varepsilon$-chaining between compact metric spaces $X$ and $\tilde{X},$ then \\
1. $\dis f \leq 10\varepsilon, \  \dis \tilde{f} \leq 10\varepsilon;$ \\
2. $\tilde{\rho}(\tilde{f}(f(\tilde{x})),\tilde{x})\leq 4\varepsilon$ holds for any $\tilde{x} \in \tilde{X}.$
\end{lemma}
\begin{proof}
Recall  
that there are maps $h, g, \tilde{h},$ and $\tilde{g}$ from the definitions that satisfy, in particular, the equalities $f= h\circ\tilde{g}$ and $\tilde{f} = \tilde{h}\circ g$ and that
$$\dis \tilde{f} = \sup\limits_{x_1, x_2 \in X} |\rho(x_1, x_2) - \tilde{\rho}(\tilde{f}(x_1), \tilde{f}(x_2))|.$$

To prove the first statement,
let us show that $|\rho(x_1, x_2) - \tilde{\rho}(\tilde{f}(x_1), \tilde{f}(x_2))|\leq 10\varepsilon$ for all $x_1, x_2\in X.$ 
Indeed, 
for all $x_1, x_2\in X,$ 
 \begin{multline*}
|\rho(x_1, x_2) - \tilde{\rho}(\tilde{f}(x_1), \tilde{f}(x_2))|  \\
\leq |\rho(x_1, x_2) - \rho(g(x_1), g(x_2))| + |\rho(g(x_1), g(x_2)) - \tilde{\rho}(\tilde{f}(x_1), \tilde{f}(x_2))|.
 \end{multline*}
Further, we have
\begin{equation}
|\rho(x_1, x_2) - \rho(g(x_1), g(x_2))|\leq \rho(x_1, g(x_1)) + \rho(x_2, g(x_2))\leq 8\varepsilon
\end{equation}
due to the following  triangle inequalities:
\begin{eqnarray*}
\rho(x_1, x_2) - \rho(g(x_1), g(x_2))&\leq& \rho(x_1, g(x_1)) + \rho(x_2, g(x_2)), \\
\rho(g(x_1), g(x_2)) - \rho(x_1, x_2) &\leq& \rho(x_1, g(x_1)) + \rho(x_2, g(x_2)),
\end{eqnarray*}
and Remark \ref{rem-ch}.
In addition,  the second summand is expressed as follows:
\begin{eqnarray*}
& & |\rho(g(x_1), g(x_2)) - \tilde{\rho}(\tilde{f}(x_1), \tilde{f}(x_2))|  \\
&=&  |\rho(g(x_1), g(x_2)) - \tilde{\rho}(\tilde{h}g(x_1), \tilde{h}g(x_2))| \\  
&\leq&  \dis \tilde{h} \leq 2\varepsilon. 
\end{eqnarray*}
Thus,
$\dis \tilde{f} \leq 10\varepsilon$ since
$|\rho(x_1, x_2) - \tilde{\rho}(\tilde{f}(x_1), \tilde{f}(x_2))|\leq 10\varepsilon$ for all $x_1$ and $x_2.$ The same inequality holds for $\dis f.$

To prove the second statement, note that, for any $\tilde{x} \in \tilde{X},$ we have  $f(\tilde{x}) = h\circ \tilde{g} (\tilde{x})\in R_{X}.$  Hence, 
$$\tilde{f}\circ f(\tilde{x}) = \tilde{h}\circ g\circ f(\tilde{x}) =\tilde{h}\circ f(\tilde{x}) = \tilde{h}\circ h\circ \tilde{g}(\tilde{x}) = h^{-1}\circ h \circ \tilde{g}(\tilde{x}) = \tilde{g}(\tilde{x})$$
since $g$ is an identical map on~$R_{X}.$
Consequently, $$\tilde{\rho}\bigl(\tilde{f}(f(\tilde{x})),\tilde{x}\bigr)=\tilde{\rho}\bigl(\tilde{g}(\tilde{x}),\tilde{x}\bigr)\leq 4\varepsilon.$$
\end{proof}

\subsection{Proof of Theorem~\ref{th-x-eps}}
\label{ssec-proof}
We will construct Lion's $\beta$-stepwise strategy in $X$ for $\beta = \sqrt{\varepsilon}.$ Since we are interested in capture by the time $T,$ it suffices to describe Lion's strategy only for  $N = \left\lceil\frac{T}{\beta}\right\rceil$ steps relying on several auxiliary constructions and Observation~\ref{obs-tuple}.

 \paragraph{A construction of Lion's strategy.}
\begin{itemize}
\item Let $(f, \tilde{f})$ be a $4\varepsilon$-chaining  between the spaces $X$ and $\tilde{X}.$
\item  Let a curve $\tilde{M}\colon [0, N\beta]\to \tilde{X}$ satisfy the condition  $\tilde{M}(0) = \tilde{f}\bigl(M(0)\bigr)$ and let it $\beta$-pursue the tuple $\{\tilde{f}(M(i\beta))\}_{i=\overline{0,..,N}}.$

Then, since the tuple satisfies the property 
$$\tilde{\rho}\bigl(\tilde{f}(M(i\beta)), \tilde{f}(M((i+1)\beta)) \bigr)\leq \rho(M(i\beta), M((i+1)\beta)) + \dis \tilde{f} \leq \beta + \dis\tilde{f}$$
for all $i = \overline{0,..,N-1},$  applying Lemma~\ref{lem-simple-pursuit} for $\delta = \frac{\dis f}{\beta},$ we obtain
\begin{eqnarray}
\label{pur-m}
\tilde{\rho}\bigl(\tilde{f}(M(i\beta)), \tilde{M}(i\beta) \bigr)\leq \beta + i\dis\tilde{f} + \tilde{\rho}(\tilde{f}(M(0)), \tilde{M}(0)) =  \beta + i\dis\tilde{f}
\end{eqnarray}
for all $i = \overline{0,..,N-1}.$

\item  By the hypothesis of this theorem, for Lion's initial position $\tilde{f}(L(0))\in \tilde{X},$ there is a Lion's strategy $\tilde{\mathfrak{s}}_{\tilde{f}(L(0))}$ leading to $\alpha$-capture by the time $T$ in the space $\tilde{X}.$ 
To use this, note that $\tilde{M}$ is a 1-Lipschitz curve, i.e., its image belongs to the set of admissible Man's trajectories in the space $\tilde{X}.$
So, let a curve $\tilde{L}\colon [0,N\beta]\to \tilde{X}$ be equal to $\tilde{\mathfrak{s}}_{\tilde{f}(L(0))}(\tilde{M})$ on the time interval $[0, T]\subset [0, N\beta]$ and, moreover, let it satisfy the condition $\tilde{L}(0) = \tilde{f}(L(0)).$ 
\item Finally, let the Lion's trajectory $L(\cdot)$ $\beta$-pursue the tuple $\{f(\tilde{L}(i\beta))\}_{i=\overline{0,..,N}}.$  
As above, we have
$$\rho\bigl(f(\tilde{L}(i\beta)), f(\tilde{L}((i+1)\beta)) \bigr)\leq \tilde{\rho}(\tilde{L}(i\beta), \tilde{L}((i+1)\beta)) + \dis f \leq \beta + \dis f$$
for all $i = \overline{0,..,N-1},$ and  by Lemma~\ref{lem-simple-pursuit} for $\delta = \frac{\dis f}{\beta},$ we get
\begin{eqnarray}
\label{pur-l}
\rho\bigl(f(\tilde{L}(i\beta)), L(i\beta)) \bigr)\leq \beta + i\dis f + \rho(f(\tilde{L}(0)), L(0)) = \beta + i\dis f
\end{eqnarray}
for all $i = \overline{0,..,N-1}.$
\end{itemize}

\paragraph{Estimates.} It remains to estimate the guaranteed capture radius. Indeed, there exists time $t^* \in [0, T]$ such that $$\tilde{\rho}(\tilde{L}(t^*), \tilde{M}(t^*))=\tilde{\rho}(\tilde{\mathfrak{s}}_{\tilde{f}(L(0))}(\tilde{M})(t^*), \tilde{M}(t^*)) \leq\alpha$$ 
because the strategy $\tilde{\mathfrak{s}}_{\tilde{f}(L(0))}$ was chosen to lead to the $\alpha$-capture. For convenience, let us choose the number $\tau$ that is nearest to $t^*$ among numbers $i\beta$ for $i=\overline{1,..,N-1};$ therefore, 
\begin{eqnarray}
\label{eq-au}
\tilde{\rho}\bigl(\tilde{L}(\tau), \tilde{M}(\tau)\bigr)\leq \alpha + 2\beta.
\end{eqnarray}
Thus, we can estimate the distance  $\rho(L(\tau), M(\tau))$ in the following way:   
$$\begin{array}{lr}
\quad \rho(L(\tau), M(\tau)) 
	&\mbox{by triangle inequality} \\  
	\leq 
\rho(L(\tau), f(\tilde{L})(\tau)) + \rho(f(\tilde{L})(\tau), M(\tau)) 
	&\mbox{by (\ref{pur-l})} \\  
	\leq 
\beta + i \dis f + \rho(f(\tilde{L})(\tau), M(\tau)) 
	&\mbox{by the definition of $\dis \tilde{f}$} \\  
	\leq 
\beta + i\dis f + \dis \tilde{f} + \tilde{\rho}\bigl(\tilde{f}(f(\tilde{L}(\tau))), \tilde{f}(M(\tau))\bigr). 
\end{array}$$
Further, let us transform $\tilde{\rho}\bigl(\tilde{f}(f(\tilde{L}(\tau))), \tilde{f}(M(\tau))\bigr)$ as follows:
$$\begin{array}{lr}
\quad \tilde{\rho}\bigl(\tilde{f}(f(\tilde{L}(\tau))), \tilde{f}(M(\tau))\bigr)  
	&\mbox{by triangle equality}  \\ 
	\leq
\tilde{\rho}\bigl(\tilde{f}(f(\tilde{L}(\tau))), \tilde{L}(\tau)\bigr) + \tilde{\rho}\bigl(\tilde{L}(\tau), \tilde{f}(M(\tau))\bigr) 
	&\mbox{by  Item~2 of Lemma~\ref{lem-f-tilf}} \\  
	\leq   
 4\varepsilon + \tilde{\rho}\bigl(\tilde{L}(\tau), \tilde{f}(M(\tau))\bigr) 
	&\mbox{by triangle inequality} \\  
	\leq   
 4\varepsilon + \tilde{\rho}\bigl(\tilde{L}(\tau), \tilde{M}(\tau))\bigr) +  \tilde{\rho}\bigl(\tilde{M}(\tau), \tilde{f}(M(\tau))\bigr)
  &\mbox{by (\ref{eq-au})} \\
    \leq   
 4\varepsilon + \alpha + 2\beta +  \tilde{\rho}\bigl(\tilde{M}(\tau), \tilde{f}(M(\tau))\bigr)
  &\mbox{by (\ref{pur-m})} \\  
  \leq   
 4\varepsilon + \alpha + 2\beta +  \beta + i\dis \tilde{f}.
\end{array}$$
 Then, recall that $\beta = \sqrt{\varepsilon}$ and
  $$i\leq N-1 = \left\lceil\frac{T}{\beta}\right\rceil-1 \leq \frac{T}{\beta}+1 - 1 = \frac{T}{\beta}. $$ 
  Moreover,  both
$\dis f$ and $\dis \tilde{f}$ are not greater than $10\varepsilon$ by Item~1 of Lemma~\ref{lem-f-tilf}; hence,
$$ i (\dis\tilde{f}+\dis f)\leq \frac{T}{\beta} 20 \varepsilon  = 20T\sqrt{\varepsilon}.$$ 
Thus,  
	\begin{eqnarray*}
\rho(L(\tau), M(\tau)) 	  &\leq& 
 \alpha + 4\beta + 4\varepsilon + i (\dis\tilde{f}+\dis f) + \dis\tilde{f}
	  \\ &\leq&
\alpha + (20T +8)\sqrt{\varepsilon}.
\end{eqnarray*} 

\bibliographystyle{spmpsci}      
\bibliography{bibfile}{}

\end{document}